\documentclass[pdflatex,sn-basic]{sn-jnl}
\usepackage{mathptmx}
\jyear{2021}%
\usepackage{titlesec}
\theoremstyle{thmstyleone}%
\newtheorem{theorem}{Theorem}
\newtheorem{corollary}[]{Corollary}
\newtheorem{assumption}{Assumption}
\newtheorem{lemma}{Lemma}
\theoremstyle{thmstyletwo}%

\theoremstyle{thmstylethree}%

\usepackage{amsmath,mathtools,amsfonts,amsmath,amssymb,mathrsfs,cases,amsthm,bm,bbm}
\newcommand{\bbP}{\mathbb{P}}
\newcommand{\E}{\mathbb{E}}
\newcommand{\RR}{\mathbb{R}}
\newcommand{\cA}{\mathcal{A}}
\newcommand{\Var}{\mathrm{var}}

\numberwithin{equation}{section}

\begin{document}

\title[Extreme Eigenvalues of Principal Minors]{Extreme Eigenvalues of Principal Minors of Random Matrix with Moment Conditions}


\author[1]{\fnm{Jianwei} \sur{Hu}}\email{jwhu@mail.ccnu.edu.cn}

\author[1]{\fnm{Seydou} \sur{Keita}}\email{badco62003@yahoo.fr}

\author*[1]{\fnm{Kang} \sur{Fu}}\email{fukang@mails.ccnu.edu.cn}

\affil[1]{\orgdiv{School of Mathematics and Statistics}, \orgname{Central China Normal University}, \orgaddress{\city{Wuhan}, \postcode{430079}, \state{Hubei}, \country{China}}}


\abstract{Let $\bm{x}_1,\cdots,\bm{x}_n$ be a random sample of size $n$ from a $p$-dimensional population distribution, where $p=p(n)\rightarrow\infty$. Consider a symmetric matrix $W=X^\top X$ with parameters $n$ and $p$, where $X=(\bm{x}_1,\cdots,\bm{x}_n)^\top$. In this paper, motivated by model selection theory in high-dimensional statistics, we mainly investigate the asymptotic behavior of the eigenvalues of the principal minors of the random matrix $W$. For the Gaussian case, under a simple condition that $m=o(n/\log p)$, we obtain the asymptotic results on maxima and minima of the eigenvalues of all $m\times m$ principal minors of $W$. We also extend our results to general distributions with some moment conditions. Moreover, we gain the asymptotic results of the extreme eigenvalues of the principal minors in the case of the real Wigner matrix. Finally, similar results for the maxima and minima of the eigenvalues of all the principal minors with a size smaller than or equal to $m$ are also given.}




\maketitle

\section{Introduction}
Random matrix theory is a popular tool in many fields, including principal component analysis, high-dimensional statistics, compressed sensing, and signal processing. In general, the random matrix theory mainly focuses on the spectral analysis of the eigenvalues and the eigenvectors of a random matrix, see, for example, \cite{Bai:1999,Johnstone:2001,Bai:2008,Bai:2010,Zou:2022}. In the past decades, the limiting laws of the extreme eigenvalues of the Wishart matrix were widely studied, such as \cite{Bai:1999,Johnstone:2001,Johnstone:2008}. Let $X$ be a $n\times p$ data matrix. Typically, each row $\bm{x}_i^\top=(x_{i1},\cdots,x_{ip})$ can be seen as a sample from a $p$-dimensional population with mean 0 and covariance matrix $\Sigma$. Under the Gaussian assumption, that is, $\bm{x}_i\sim N(0,\Sigma)$, we call $W=X^\top X$ a Wishart matrix. Further, if $\Sigma=I_p$, we call $W$ a white Wishart matrix. Let $\lambda_1(W)\geq\cdots\geq\lambda_p(W)$ be the eigenvalues of $W$. When $n,p\rightarrow\infty$ and $n/p\rightarrow \gamma$, \cite{Johnstone:2001} gave the following asymptotic result:
\[
\dfrac{\lambda_1(W)-\mu_{np}}{\sigma_{np}}\stackrel{d}{\longrightarrow}TW_1,
\]
where $\mu_{np}=(\sqrt{n-1}+\sqrt{p})^2, \sigma_{np}=(\sqrt{n-1}+\sqrt{p})\left(1/\sqrt{n-1}+1/\sqrt{p}\right)^{1/3}$, $TW_1$ is the Tracy-Widom distribution with index 1, and we use ``$\stackrel{d}{\longrightarrow}$" to indicate convergence in distribution. Similarly, the limiting distribution of $\lambda_p(W)$ was established, see, for example, \cite{Edelman:1988,Bai:1993}. In addition to the Wishart matrix, the Wigner matrix also plays an important role in the random matrix theory. For a real matrix $W$, each entry $w_{ij}$ is a real normal random variable, then matrix $W$ is called the Wigner matrix. When $w_{ij}$'s follow the standard normal distribution, \cite{Tracy:1994} established the following asymptotic result:
\[
n^{2/3}(\lambda_1(W)-2)\stackrel{d}{\longrightarrow}TW_1.
\]
These results were also extended to the general case, this is, the entries of the matrix follow the general distribution (may not be normal distribution), see, for example, \cite{Bai:2010,Tao:2010}.

In fact, for a random sample of size $n$, $\bm{x}_1,\cdots, \bm{x}_n$, the sample covariance matrix can be obtained by dividing the matrix $W$ by $n$, i.e., $\Sigma=\dfrac{1}{n}X^\top X=\dfrac{1}{n}\sum\limits_{i=1}^n\bm{x}_i^\top\bm{x}_i$. The sample covariance matrix is fundamental to multivariate statistical inference. Meanwhile, the eigenvalues of the sample covariance matrix play a critical role in a hypothesis test, principal components analysis, factor analysis, and discrimination analysis. \cite{Geman:1980} first proved that the largest eigenvalue of sample covariance matrix tends to $\sigma^2(1+\sqrt{\gamma})^2$, where $\sigma^2$ is the variance of the entries of $X$, when $p/n\rightarrow\gamma\in(0, +\infty)$ under some moment conditions. This work was generalized by \cite{Bai:1988} and \cite{Yin:1988} under the assumption of the existence of the fourth moment. \cite{Lee:2016} proved that the largest eigenvalue of the real sample covariance matrix follows the Tracy-Widom distribution in general population cases. In some literature, the spiked model is also widely considered. \cite{Baik:2006} proved the limit of the eigenvalues of large sample covariate matrix in spiked population models. \cite{Bai:2008} established the central limit theorem (CLT) for all spiked eigenvalues of the sample covariance matrix under the spiked population model.

Motivated by variable selection in high-dimensional statistics, we investigated the extreme eigenvalues of the principal minors of a random matrix. Consider the general linear regression model
\begin{equation}\label{eq:lm}
	y=X\beta+\varepsilon,
\end{equation}
where $y\in\RR^n,\beta\in\RR^p,X\in\RR^{n\times p}$ with $n\ll p$, and $\varepsilon\sim N(0,\sigma^2I_n)$. Let $\cA=\{i:\beta_i\neq0\}$ and $\vert\cA\vert=\#\{i:\beta_i\neq0\}$, the purpose of the model selection is to obtain an estimator $\hat{\beta}$ such that $\bbP(\hat{\cA}=\cA)\rightarrow1$ with $\hat{\cA}=\{i:\hat{\beta}_i\neq0\}$, i.e., the selection consistency. To obtain the estimator, a widely used method is the penalty likelihood method, this is,
$$\hat{\beta}=\arg\min_{\beta}\dfrac{1}{2n}\|y-X\beta\|^2+\sum_{j=1}^p P_{\lambda}(\vert\beta_j\vert),$$
where $P_{\lambda}(\vert\beta_j\vert)$ is a penalty function indexed by $\lambda>0$. A widely used penalized function is the LASSO penalty \citep{Tibshirani:1996}. Although the LASSO estimator is easy to be obtained, the LASSO estimator is biased. \cite{Fan:2001} showed that the bias of the LASSO estimator can be eliminated by choosing the SCAD penalty. Further, under the minimax concave penalty, \cite{Zhang:2010} proposed an MC+ method, which is a fast and nearly unbiased concave penalized selection method in the model \eqref{eq:lm}. In \cite{Zhang:2010}, a critical condition is the \textit{sparse Riesz condition} (SRC). For $S\subset\{1,\ldots,p\}$, the sub-design and the sub-Gram matrices are defined as follows:
\[
X_S=(\textbf{x}_j,j\in S)_{n\times\vert S\vert}\ \text{and}\ \Sigma_{S}=\dfrac{1}{n}X_S^\top X_S,
\]
where $\textbf{x}_j$ is the $j$th column of the matrix $X$. The SRC assumes that for suitable $0<c_1\leq c_2<\infty$ and a constant $m$,
\[
c_1\leq\min_{\vert S\vert \leq m}\lambda_{\mathrm{min}}(\Sigma_S)\leq\max_{\vert S\vert \leq m}\lambda_{\mathrm{max}}(\Sigma_S)\leq c_2,
\]
where $\lambda_{\mathrm{min}}(\Sigma_S)$ and $\lambda_{\mathrm{max}}(\Sigma_S)$ are the smallest and the largest eigenvalues of $\Sigma_S$, respectively. Here, $\Sigma_S$ can be seen as a principal minor of the sample covariance $\Sigma$. Under the SRC, \cite{Zhang:2010} proved that the selection by the MC+ method is nearly unbiased and consistent, i.e.,
\[
\bbP\left(\hat{\cA}=\cA\right)\rightarrow1,\ \text{and}\ \|\hat{\beta}-\beta\|_q^q=O_p(1).
\]

Hence, the main object of interest in this paper is the extreme eigenvalues of the principal minors of a Wishart matrix $W=X^\top X$, that is, $$\max_{S\subset\{1,\ldots,p\},\vert S\vert=m}\lambda_1(W_S),$$ and $$\min_{S\subset\{1,\ldots,p\},\vert S\vert=m}\lambda_m(W_S),$$ where $W_S$ is a $m\times m$ principal minor of $W$, and $\lambda_1(W_S)$ and $\lambda_m(W_S)$ are the largest and the smallest eigenvalues of $W_S$, respectively.

In the case where $p$ and $n$ are of the same order, that is, $n/p\rightarrow\gamma\in(0,+\infty)$, the asymptotic properties of the extreme eigenvalues of the Wishart matrix $W$ were extensively studied recently, see, for example, \cite{Bai:1999,Johnstone:2001,Johnstone:2008}. We refer to \cite{Cai:2021} and references therein for recent developments on this topic. In particular, \cite{Cai:2021} considered the normal case where the entries $x_{ij}$'s of $X$ are independent and identically distributed (i.i.d.) $N(0,1)$ variables. They proved the following main results.

\textit{ Suppose the integer $m\geq 1$ is fixed and $\log p = o(n^{1/2})$; or $m\rightarrow\infty$ with $$m=o\left(\min\left\{\dfrac{(\log p)^{1/3}}{\log\log p},\dfrac{n^{1/4}}{(\log n)^{3/2}(\log p)^{1/2}}\right\}\right).$$ Assume $p=p(n)\rightarrow\infty$ and $p/n\rightarrow\gamma\in(0,\infty)$. Define $$T_{m,n,p}=\max_{S\subset\{1,\ldots,p\},\vert S\vert=m}\lambda_1(W_S).$$ Then,
\begin{equation}
	Z_n:=\dfrac{T_{m,n,p}-n}{\sqrt{n}}-2\sqrt{m\log p}\rightarrow0
\end{equation}
in probability as $n\rightarrow\infty$.
}

Some similar results were also given in \cite{Cai:2021}. It is easy to see that the asymptotic behavior of $T_{m,n,p}$ depends on complex assumptions, and they only considered the case of standard normality. In this paper, based on a simple condition that $m=o(n/\log p)$, we investigate the limiting behavior for the extreme eigenvalues of the Wishart matrix. This condition relaxes the condition in \cite{Cai:2021}. Meanwhile, our results do not depend on a basic condition that $p/n\rightarrow\gamma\in(0,\infty)$, which is required in the literature about random matrix theory. One key step in the proofs of our results is that we give a finer upper bound of the matrix spectral norm (see, Lemma \ref{matrixbound}). Since the upper bound depends on a quadratic form, we transform the quadratic form of the matrix into a sum of independent and identical distribution random variables. Hence, the problem of eigenvalues becomes the problem of the sum of independent and identically distributed random variables. Compared with the results in \cite{Cai:2021}, we extend the distribution of the entries of the sample matrix $X$ to general distributions with mean 0, variance 1, and finite fourth moments, and obtain the asymptotic results of the extreme eigenvalues of $X^\top X$ under some moment conditions. Meanwhile, as a natural by-product, we also consider the corresponding results when $W$ is a real Wigner matrix.

Throughout the paper, the following notions will be used. For a finite set $S$, we use $\#(S)$ or $\vert S\vert$ to denote the cardinality of the set $S$. For a matrix $A\in\RR^{n\times n}$, we denote the spectral norm by $\|A\|$. For two positive sequences $\{a_n\}$ and $\{b_n\}$, we write $a_n=o(b_n)$ if $\lim_{n\rightarrow\infty}a_n/b_n=0$. Further, for a sequence of random variables $X_n$ and a positive sequence $a_n$, we write $X_n=O_p(a_n)$ if for all $\varepsilon>0$, there is an $M$ such that $\sup_{n}\bbP(\vert X_n/a_n\vert>M)<\varepsilon$.

The rest of the paper is organized as follows. Section \ref{Section2} describes the precise setting of the problem. The main results, Theorems \ref{T1}, \ref{T2}, and \ref{T3} are stated in Section \ref{Section3}. Some related corollaries are also considered in Section \ref{Section3}. The proofs of the main theorems are given in Section \ref{Section4}.

\section{Problem Setting}\label{Section2}
In this section, we formally state the issue of our study. Let $X$ be a $n\times p$ matrix whose generic elements, $x_{ij}$'s, are independent and identically distributed random variables. In this paper, we mainly focus on two cases: the standard normal distribution and the general distributions with mean 0, variance 1, and some moment conditions. Then, $W=X^\top X$ is a white Wishart matrix when $x_{ij}$'s are the standard normal random variables. First, we give some notations.  Let $S\subset\{1,\cdots,p\}$ and $W_S=X_S^\top X_S=(w_{ij})_{i,j\in S}$, where $X_S$ is an $n\times\vert S\vert$ matrix. Hence, $W_S$ is a principal minor of $W$. Define
\begin{equation}\label{WS1}
	T_{m,n,p}=\max_{S\subset\{1,\ldots,p\},\vert S\vert=m}\lambda_1(W_S),
\end{equation}
and
\begin{equation}\label{WSm}
	V_{m,n,p}=\min_{S\subset\{1,\ldots,p\},\vert S\vert=m}\lambda_m(W_S).
\end{equation}
Further, let $A=\dfrac{1}{\sqrt{n}}(X^\top X-nI)$ and $A_S=\dfrac{1}{\sqrt{n}}(X_S^\top X_S-nI)$, where $I$ indicates the identical matrix, and its dimension depends on the specific equation and may vary from equation to equation. Similarly, we also define
\begin{equation}\label{AS1}
	\lambda_{\max}(m)=\max_{S\subset\{1,\ldots,p\},\vert S\vert=m}\lambda_1(A_S),
\end{equation}
and
\begin{equation}\label{ASm}
	\lambda_{\min}(m)=\min_{S\subset\{1,\ldots,p\},\vert S\vert=m}\lambda_m(A_S).
\end{equation}
Later, $\lambda_1(A_S)$ and $\lambda_m(A_S)$ will play a critical role in the proof of the asymptotic behavior of $\lambda_1(W_S)$ and $\lambda_m(W_S)$. We obtain an upper bound of the spectral norm of $A$ by applying $\varepsilon$-net argument (see, Lemma \ref{matrixbound}). According to the properties of eigenvalues, it is easy to know that
\begin{equation*}
	T_{m,n,p}=n+\sqrt{n}\lambda_{\max}(m)\ \text{and}\ V_{m,n,p}=n+\sqrt{n}\lambda_{\min}(m).
\end{equation*}
In this paper, the main interest is the asymptotic behavior of the statistics $T_{m,n,p}$ and $V_{m,n,p}$ when both $n$ and $p$ grow large for different distributed cases. Therefore, we can gain the asymptotic results of $T_{m,n,p}$ and $V_{m,n,p}$ by studying the asymptotic behavior of $\lambda_{\max}(m)$ and $\lambda_{\min}(m)$.

\section{Main Results}\label{Section3}

This section considers the laws of large numbers of $T_{m,n,p}$ and $V_{m,n,p}$ for three different random variable cases. Throughout the paper, we let $n\rightarrow\infty$ and let $p\rightarrow\infty$ with a rate depending on $n$. First, we give the following assumption. This assumption is a critical condition in our main results.

\begin{assumption}\label{A1}
The integer $m$ satisfies that
\begin{equation}\label{Aeq1}
	m=o(n/\log p).
\end{equation}
\end{assumption}
Note that Assumption \ref{A1} implies that $\dfrac{m\log p}{n}=o(1)$. This condition is mainly used in the analysis of $\lambda_{\max}(m)$ and $\lambda_{\min}(m)$. It is easy to see that Assumption \ref{A1} relaxes the condition that $m=o\left(\min\left\{\dfrac{(\log p)^{1/3}}{\log\log p},\dfrac{n^{1/4}}{(\log n)^{3/2}(\log p)^{1/2}}\right\}\right)$ in \cite{Cai:2021}. Without loss of generality, we assume $S=\{1,\ldots,m\}$ for $\vert S\vert=m$. Note that the upper bound of $\|A\|$ depends on a quadratic form of $A$. Then, for any unit vector $u=(u_1,\dots,u_m)^\top\in\RR^m$, we have $$u^\top A_Su=\frac{1}{\sqrt{n}}\sum_{t=1}^n[\sum_{i=1}^m(x_{ti}^2-1)u_i^2+2\sum_{i<j}x_{ti}x_{tj}u_iu_j]=:\frac{1}{\sqrt{n}}\sum_{t=1}^n\xi_t,$$ where
    \begin{align}\label{eq:xi}
	\xi_t &=\sum_{i=1}^m(x_{ti}^2-1)u_i^2+2\sum_{i<j}x_{ti}x_{tj}u_iu_j\notag\\
    &=(\sum_{i=1}^mu_ix_{ti})^2-1\notag\\
    &=\zeta_t^2-1,
    \end{align}
    and $\zeta_t=\sum_{i=1}^mu_ix_{ti}$. Here, we transform the quadratic form of a matrix into the sum of independent random variables. Under the different distribution assumptions, $\xi_t$'s have different properties. Hence, in this paper, we mainly focus on the Gaussian distribution and the general distribution with some moment conditions.

\subsection{The Gaussian case}
First, we consider that the entries $x_{ij}$'s of $X$ are i.i.d normal random variables with mean 0 and variance 1, that is, $x_{ij}\sim N(0,1)$ for any $1\leq i\leq n,1\leq j\leq p$. In this case, $W=X^\top X$ can be seen as a white Wishart matrix. According to \eqref{eq:xi}, it is easy to show that $\zeta_t\sim N(0,1)$. On the other hand, it is known that $\xi_1,\cdots,\xi_n$ are i.i.d. random variables with $\E(\xi_t)=0$ and $$\E(\xi_t^2)=\sum_{i=1}^m\E(x_{ti}^2-1)^2u_i^4+4\sum_{i<j}u_i^2u_j^2=2.$$ We start with asymptotic results for $T_{m,n,p}$ in (\ref{WS1}) and $V_{m,n,p}$ in \eqref{WSm}. The next theorem gives the result of $T_{m,n,p}$.

\begin{theorem}\label{T1}
	Suppose Assumption \ref{A1} holds and $x_{ij}$'s are standard normal random variables. Recall $T_{m,n,p}$ defined as in (\ref{WS1}). Then
	\begin{equation*}
		\E\left(e^{t_0\vert \xi_1\vert}\right)<\infty\ \mathrm{for}\ 0<t_0<\dfrac12
	\end{equation*} and
	\begin{equation}\label{normal:T}
		\dfrac{T_{m,n,p}}{n}=1+O_p\left(2\sqrt{\dfrac{m\log p}{n}}\right).
	\end{equation}
\end{theorem}

{\bf Remark 1.} Suppose Assumption \ref{A1} holds and $x_{ij}$'s are standard normal random variables. Recall $V_{m,n,p}$ defined as in (\ref{WSm}). Similar to the results of Theorem \ref{T1}, it can be shown that
	\begin{equation}\label{normal:V}
		\dfrac{V_{m,n,p}}{n}=1+O_p\left(2\sqrt{\dfrac{m\log p}{n}}\right).
	\end{equation}

The proof of Theorem \ref{T1} is given in Section \ref{Section4}. According to the details of the proof, we have $$\lim_{n\to\infty} \bbP\left(1-2\sqrt{\frac{m\log p}{n}}\leq\frac{V_{m,n,p}}{n}\leq\frac{T_{m,n,p}}{n}\leq 1+2\sqrt{\frac{m\log p}{n}}\right)=1,$$ which implies that it holds with probability approaching 1 that $$1-2\sqrt{\frac{m\log p}{n}}\leq\min_{\vert S\vert=m}\lambda_m(X_S^\top X_S/n)\leq\max_{\vert S\vert=m}\lambda_1(X_S^\top X_S/n)\leq 1+2\sqrt{\frac{m\log p}{n}},$$ which is a special form of SRC.

We now consider a similar extension for the above results. From the view of SRC, the limiting behavior of the eigenvalues of all principal minors with the size smaller than or equal to $m$ is also important. It means that we should consider the following statistics:
\[\check{T}_{m,n,p}=\max_{S\subset\{1,\ldots,p\},\vert S\vert \leq m}\lambda_1(W_S)\ \text{and}\ \check{V}_{m,n,p}=\min_{S\subset\{1,\ldots,p\},\vert S\vert \leq m}\lambda_m(W_S).
\]
The next corollary indicates that Theorem \ref{T1} still holds if we replace the principal minors with the size $m$ by the principal minors with the size smaller than or equal to $m$ in the previous results.

\begin{corollary}\label{Coro1}
Theorem \ref{T1} still holds if ``$T_{m,n,p}$" is replaced by ``$\check{T}_{m,n,p}$". Similarly, $\check{V}_{m,n,p}$ also satisfies \eqref{normal:V} if ``$V_{m,n,p}$" is replaced by ``$\check{V}_{m,n,p}$".
\end{corollary}

The results of Corollary \ref{Coro1} imply that it holds with probability approaching 1 that $$1-2\sqrt{\frac{m\log p}{n}}\leq\min_{\vert S\vert\leq m}\lambda_m(X_S^\top X_S/n)\leq\max_{\vert S\vert\leq m}\lambda_1(X_S^\top X_S/n)\leq 1+2\sqrt{\frac{m\log p}{n}}.$$ It is a desirable result and is consistent with the SRC.

\subsection{The general case}

In variable selection, the entries of the sample matrix $X$ may not follow the normal distribution. Hence, a related problem is whether Theorem \ref{T1} can be extended to non-Gaussian $x_{ij}$. \cite{Cai:2021} only conjectured that the asymptotic behavior of $T_{m,n,p}$ with non-Gaussian variables will be similar to that of $T_{m,n,p}$ as discussed in Theorem \ref{T1} under certain assumptions on the moments of $x_{ij}$. Next, on the condition $m=o(n/\log p)$, we get asymptotic behavior of $T_{m,n,p}$ under some moment conditions.

\begin{assumption}\label{A2}
	The entries of $X$ are i.i.d. with mean 0 and variance 1. Assume that $\Var(x_{ij}^2)=\eta>0$ for $1\leq i\leq n$, $1\leq j\leq p$.
\end{assumption}

Compared with the condition of Theorem \ref{T1}, Assumption \ref{A2} allows the entries of $X$ to be the general distribution with mean 0 and variance 1. Further, the condition that $\Var(x_{ij}^2)=\eta>0$ implies that the distribution has finite fourth moments. Similarly, we consider the asymptotic behavior of $T_{m,n,p}$ and $V_{m,n,p}$. Next, we have the following results:

\begin{theorem}\label{T2}
	Suppose Assumptions \ref{A1} and \ref{A2} hold. Assume that $\sup\limits_{\|u\|=1}\E \mathrm{e}^{t_0\vert \xi_1\vert}<\infty$ for some $t_0>0$, where $\xi_1$ has the form of \eqref{eq:xi}. Recall $T_{m,n,p}$ defined as in (\ref{WS1}), then,
	\begin{enumerate}
		\item[(i)] when $0<\eta\leq2$,
		\begin{equation}\label{normal:T1}
		\dfrac{T_{m,n,p}}{n}=1+O_p\left(\sqrt{\frac{[4(m-1)+2\eta]\log p}{n}}\right).
	    \end{equation}
	    \item[(ii)] when $\eta>2$,
	    \begin{equation}\label{normal:T2}
		\dfrac{T_{m,n,p}}{n}=1+O_p\left(\sqrt{\frac{2\eta m\log p}{n}}\right).
	    \end{equation}
	\end{enumerate}
\end{theorem}

{\bf Remark 2.}  Recall $V_{m,n,p}$ defined as in (\ref{WSm}). Under the conditions of Theorem \ref{T2}, it can be shown that,
	\begin{enumerate}
		\item[(i)] when $0<\eta\leq2$,
		\begin{equation}\label{normal:V1}
		\dfrac{V_{m,n,p}}{n}=1+O_p\left(\sqrt{\frac{[4(m-1)+2\eta]\log p}{n}}\right).
	    \end{equation}
	    \item[(ii)] when $\eta>2$,
	    \begin{equation}\label{normal:V2}
		\dfrac{V_{m,n,p}}{n}=1+O_p\left(\sqrt{\frac{2\eta m\log p}{n}}\right).
	    \end{equation}
	\end{enumerate}

Similar to the case where $x_{ij}$'s are normal random variables, we also consider the asymptotic results when the size of the principal minors is no larger than $m$. We then have the following corollary.

\begin{corollary}\label{Coro2}
Recall $\check{T}_{m,n,p}$ and $\check{V}_{m,n,p}$. Then Theorem \ref{T2} still holds if ``$T_{m,n,p}$'s" are replaced by ``$\check{T}_{m,n,p}$'s" for $0< \eta\leq 2$ and $\eta>2$. Similarly, $\check{V}_{m,n,p}$ also satisfies \eqref{normal:V1} and \eqref{normal:V2} if ``$V_{m,n,p}$'s" are replaced by ``$\check{V}_{m,n,p}$'s" for $0< \eta\leq 2$ and $\eta>2$.
\end{corollary}

Similar to the Gaussian case, Corollary \ref{Coro2} implies that we can also obtain a result that is consistent with the SRC.

\subsection{Wigner matrix case}

Notice that $w_{ij}=\sum_{k=1}^nx_{ki}x_{kj}$. Since $x_{ij}$'s are i.i.d random variables, $w_{ij}$ is the sum of the i.i.d. random variables. When $x_{ij}$'s are independent standard normal variables, we have $$\E\{w_{ij}\}=\begin{dcases}0,& \mathrm{if}\ i=j,\\ n,&\mathrm{if}\ i\neq j,\end{dcases}$$ and $\Var(w_{ij})=n$ for any $1\leq i,j\leq p$. By the standard CLT, for given $1\leq i,j\leq p$, we have the following results:
$$\dfrac{w_{ij}-n}{\sqrt{n}}\stackrel{d}{\longrightarrow}N(0,2)\ \text{if}\ i=j,\ \text{and}\ \dfrac{w_{ij}}{\sqrt{n}}\stackrel{d}{\longrightarrow}N(0,1)\ \text{if}\ i\neq j,$$ as $n\rightarrow\infty$. These limiting results motivate us to consider the case of the Wigner matrix. Let $\tilde{W}=(\tilde{w}_{ij})_{1\leq i,j\leq p}$ be a Wigner matrix, which is a symmetric matrix whose all elements follow the normal distribution. Specifically, for $\eta>0$, we assume
\begin{equation}\label{Wigner}
	\tilde{w}_{ij}\sim\begin{dcases} N(0,\eta) & \text{if}\ i=j; \\ N(0,1) & \text{if}\ i<j. \end{dcases}
\end{equation}

For $S\subset\{1,2,\cdots,p\}$, we denote $\tilde{W}_S=(\tilde{w}_{ij})_{i,j\in S}$. In this subsection, we will investigate the following two statistics:
\begin{equation}\label{tW1}
	\tilde{T}_{m,p}=\max_{S\subset\{1,\cdots,p\},\vert S\vert=m}\lambda_1(\tilde{W}_S)
\end{equation}
and
\begin{equation}\label{tWm}
	\tilde{V}_{m,p}=\min_{S\subset\{1,\cdots,p\},\vert S\vert=m}\lambda_m(\tilde{W}_S).
\end{equation}

When $0<\eta\leq2$, \cite{Cai:2021} gave the asymptotic behavior of $\tilde{T}_{m,p}$ and $\tilde{V}_{m,p}$ under a complicated condition. However, the result of the case when $\eta>2$ was not given. Here, we want to study the case of $\eta>2$. Hence, for statistics $\tilde{T}_{m,p}$ and $\tilde{V}_{m,p}$, the following laws of large numbers are obtained.

\begin{theorem}\label{T3}
Suppose $p\rightarrow\infty$. In addition, assume the entries of $\tilde{W}$ have the distribution as in \eqref{Wigner}. Then,
\begin{enumerate}
	\item[(i)] when $0<\eta\leq2$,
	\begin{equation}
		\lim_{p\to\infty} \bbP\left(-1\leq\frac{\tilde{V}_{m,p}}{\sqrt{[4(m-1)+2\eta]\log p}}\leq\frac{\tilde{T}_{m,p}}{\sqrt{[4(m-1)+2\eta]\log p}}\leq 1\right)=1.
	\end{equation}
	
	\item[(ii)] when $\eta>2$,
	\begin{equation}
		\lim_{p\to\infty} \bbP\left(-1\leq\frac{\tilde{V}_{m,p}}{\sqrt{2\eta m\log p}}\leq\frac{\tilde{T}_{m,p}}{\sqrt{2\eta m\log p}}\leq 1\right)=1.
	\end{equation}

\end{enumerate}
\end{theorem}

Notice that the results in Theorem \ref{T3} give the asymptotic upper bound and lower bound of the statistics $\tilde{T}_{m,p}$ and $\tilde{V}_{m,p}$. Compared with the result in \cite{Cai:2021}, we not only relax the condition, but also obtain the result when $\eta>2$.

\section{Technical Proofs}\label{Section4}

Before proving the main Theorems, we need the following two key lemmas.
\begin{lemma}\label{matrixbound}
For any $m\times m$ symmetric matrix $A$, there exist $v_1,\cdots,v_{(1+2/\varepsilon)^m}\in S^{m-1}$ such that the following inequality holds $$
\|A\|\leq \frac{1}{1-\sqrt{\varepsilon^2(4-\varepsilon^2)}}\sup_{j\leq(1+2/\varepsilon)^m} \vert v_j^\top Av_j\vert,$$ for $\varepsilon\in(0,\sqrt{2-\sqrt{3}})$, where $S^{m-1}\subset \mathbb{R}^m$ is unit sphere in the Euclidean distance.
\end{lemma}

\begin{proof}
For any $m\times m$ symmetric matrix $A$ and unit vectors $u,v$, we have
\begin{align*}
	\vert u^\top Au\vert-\vert v^\top Av\vert  & \leq \vert u^\top Au-v^\top Av\vert \\
	& = \vert (u-v)^\top A(u+v)\vert \\
	& \leq \|u-v\|\cdot\|A\|\cdot\|u+v\|.
\end{align*}

For $\varepsilon\in(0,\sqrt{2-\sqrt{3}})$, let $\|u-v\|=t\leq\varepsilon$. Note that $\|u-v\|^2 = 2-2u^\top v$, then we have $\|u+v\|^2 = 2+2u^\top v=4-t^2$. Hence, $\|u-v\|\|u+v\|=\sqrt{t^2(4-t^2)}$. Since the right of inequality is an increasing function about $t$ when $0<t<\sqrt{2}$, we have $\|u-v\|\|u+v\|\leq\sqrt{\varepsilon^2(4-\varepsilon^2)}$. Thus, $$\vert u^\top Au\vert-\vert v^\top Av\vert\leq\varepsilon\|A\|\sqrt{4-\varepsilon^2}.$$

Let $S^{m-1}_{\varepsilon}$ be an $\varepsilon$-net of the unit sphere $S^{m-1}\subset \mathbb{R}^m$ in the Euclidean distance. We have
$$\|A\| \leq \sup_{u\in S^{m-1}}\vert u^\top Au\vert \leq\sup_{v\in S^{m-1}_{\varepsilon}}\vert v^\top Av\vert+\|A\|\sqrt{\varepsilon^2(4-\varepsilon^2)},$$
which implies that, for $0<\varepsilon<\sqrt{2-\sqrt{3}}$,
$$\|A\|\leq \frac{1}{1-\sqrt{\varepsilon^2(4-\varepsilon^2)}}\sup_{v\in S^{m-1}_{\varepsilon}}\vert v^\top Av\vert.$$
Since we are allowed to pack $\#(S^{m-1}_{\varepsilon})$ balls of radius $\varepsilon/2$ into a $1+\varepsilon/2$ ball in $\mathbb{R}^m$, volume comparison yields
$$(\varepsilon/2)^m\#(S^{m-1}_{\varepsilon})\leq(1+\varepsilon/2)^m,$$
that is, $$\#(S^{m-1}_{\varepsilon})\leq(1+2/\varepsilon)^m.$$

There exist $v_1,\dots,v_{(1+2/\varepsilon)^m}\in S^{m-1}$ such that
$$\|A\|\leq \frac{1}{1-\sqrt{\varepsilon^2(4-\varepsilon^2)}}\sup_{j\leq(1+2/\varepsilon)^m}\vert  v_j^\top Av_j\vert,$$
for any $m\times m$ symmetric matrix $A$.
\end{proof}

\begin{lemma}
\label{lemma:chen}\citep{Chen:1990}
Suppose $\xi_1,\xi_2,\ldots,\xi_n$ are i.i.d random variables with $\E\xi_1=0$ and $\E\xi_1^2=1$. Set $S_n=\sum_{i=1}^n\xi_i$. Let $0<\alpha\leq 1$ and $\{a_n:n\geq 1\}$ satisfy that $a_n\rightarrow\infty$ and $a_n=o(n^{\alpha/(2(2-\alpha))})$. If $\E e^{t_0\vert \xi_1\vert^\alpha}<\infty$ for some $t_0>0$, then
$$\lim_n\frac{1}{a_n^2}\log \bbP(\frac{S_n}{\sqrt{n}a_n}\geq \mu)=-\frac{\mu^2}{2}$$
for any $\mu>0$.
\end{lemma}

\subsection{Proof of Theorem \ref{T1}}

It is easy to know that $\vert\xi_t\vert=\vert\zeta_t^2-1\vert<\max\{\zeta_t^2,1\}$. Then we have
	\begin{align*}
		\E(\vert\xi_t\vert^k)&<\E(\max\{\zeta_t^{2k},1\})\\
        &\leq\E(\zeta_t^{2k})+1\\
        &=(2k-1)!!+1.
	\end{align*}
	
For $0<t_0<\dfrac{1}{2}$, it holds that

\begin{align*}
\E(e^{t_0\mid\xi_1\mid})&=\E(\sum_{k=0}^{\infty}\frac{t_0^k\mid\xi_1\mid^k}{k!})\\
&=\sum_{k=0}^{\infty}\E(\frac{t_0^k\mid\xi_1\mid^k}{k!})\\
&\leq\sum_{k=0}^{\infty}(2t_0)^k\frac{(k-\frac{1}{2})(k-\frac{3}{2})\cdots\frac{1}{2}+2^{-k}}{k!}\\
&\leq\sum_{k=0}^{\infty}(2t_0)^k\\
&=\frac{1}{1-2t_0}.
\end{align*}

Recall that $u^\top A_Su=\frac{1}{\sqrt{n}}\sum_{t=1}^n\xi_t$, where $\xi_1,\cdots,\xi_n$ are i.i.d. random variables with $\E(\xi_t)=0$ and $\E(\xi_t^2)=2$. By Lemmas \ref{matrixbound} and \ref{lemma:chen}, when $x=o(\sqrt{n})$, we have
\begin{align*}
	P(\|A_S\|\geq x)&\leq P\left(\dfrac{1}{1-\sqrt{\varepsilon^2(4-\varepsilon^2)}}\sup_{j\leq(1+2/\varepsilon)^m}\vert v_j^\top A_S v_j\vert \geq x\right)\\
&=P\left(\sup_{j\leq(1+2/\varepsilon)^m}\mid v_j^\top Av_j\mid\geq \left[1-\sqrt{\varepsilon^2(4-\varepsilon^2)}\right]x\right)\\
&\leq 2\exp\left\{m\log(1+2/\varepsilon)-\frac{[1-\sqrt{\varepsilon^2(4-\varepsilon^2)}]^2x^2}{4}\right\},
\end{align*}
for $\varepsilon\in(0,\sqrt{2-\sqrt{3}})$ and sufficiently large $n$.

We first prove
\begin{equation}\label{eq:proof1}
	\lim_{n\to\infty} \bbP\left(\frac{\lambda_{\max}(m)}{2\sqrt{m\log p}}\geq 1+\delta\right)=0,
\end{equation}
for any $\delta>0$ small enough.

Let $z_n=2(1+\delta)\sqrt{m\log p}$. It is easy to see that $z_n=o(\sqrt{n})$. Hence, we have
\begin{align*}
&\ \bbP\left(\frac{\lambda_{\max}(m)}{2\sqrt{m\log p}}\geq 1+\delta\right)\\
=&\ \bbP\left(\frac{\max\limits_{S\subset\{1,\ldots,p\},\vert S\vert=m}\lambda_1(A_S)}{2\sqrt{m\log p}}\geq 1+\delta\right)\\
\leq&\ \sum_{S\subset\{1,\ldots,p\},\vert S\vert =m}\bbP\left(\frac{\lambda_1(A_S)}{2\sqrt{m\log p}}\geq 1+\delta\right)\\
\leq&\ \sum_{S\subset\{1,\ldots,p\},\vert S\vert =m}\bbP\left(\frac{\mid\mid A_S\mid\mid}{2\sqrt{m\log p}}\geq 1+\delta\right)\\
\leq&\ 2\exp\left\{m\log p+m\log(1+2/\varepsilon)-\frac{[1-\sqrt{\varepsilon^2(4-\varepsilon^2)}]^2z_n^2}{4}\right\}\\
=&\ 2\exp\left\{[1-(1+\delta)^2(1-\sqrt{\varepsilon^2(4-\varepsilon^2)})^2]m\log p+m\log(1+2/\varepsilon)\right\}
\end{align*}
for sufficiently large $n$. Note that it holds that $$1-(1+\delta)^2[1-\sqrt{\varepsilon^2(4-\varepsilon^2)}]^2<0$$ when $$\varepsilon^2<2-\sqrt{4-(1-\frac{1}{1+\delta})^2}.$$ Hence, we have $$\bbP\left(\frac{\lambda_{\max}(m)}{2\sqrt{m\log p}}\geq 1+\delta\right)=o(1).$$
	
Next, we also need to show that for any $\delta>0$,
\begin{equation}\label{eq:proof2}
	\lim_{n\to\infty} \bbP\left(\frac{\lambda_{\min}(m)}{2\sqrt{m\log p}}\leq -1-\delta\right)=0.
\end{equation}

Then, similarly to \eqref{eq:proof1}, we have
\begin{align*}
&\ \bbP\left(\frac{\lambda_{\min}(m)}{2\sqrt{m\log p}}\leq -1-\delta\right)\\
=&\ \bbP\left(\frac{-\lambda_{\min}(m)}{2\sqrt{m\log p}}\geq 1+\delta\right)\\
=&\ \bbP\left(\frac{\max\limits_{S\subset\{1,\ldots,p\},\vert S\vert =m}(-\lambda_m(A_S))}{2\sqrt{m\log p}}\geq 1+\delta\right)\\
\leq&\ \sum_{S\subset\{1,\ldots,p\},\vert S\vert =m}\bbP\left(\frac{-\lambda_m(A_S)}{2\sqrt{m\log p}}\geq 1+\delta\right)\\
\leq&\ \sum_{S\subset\{1,\ldots,p\},\vert S\vert =m}\bbP\left(\frac{\mid\mid A_S\mid\mid}{2\sqrt{m\log p}}\geq 1+\delta\right)\\
\leq&\ 2\exp\left\{m\log p+m\log(1+2/\varepsilon)-\frac{[1-\sqrt{\varepsilon^2(4-\varepsilon^2)}]^2z_n^2}{4}\right\}\\
=&\ 2\exp\left\{[1-(1+\delta)^2(1-\sqrt{\varepsilon^2(4-\varepsilon^2)})^2]m\log p+m\log(1+2/\varepsilon)\right\}
\end{align*}
for sufficiently large $n$. Similarly, when $\varepsilon^2<2-\sqrt{4-(1-\frac{1}{1+\delta})^2}$, we have $$\bbP\left(\frac{\lambda_{\min}(m)}{2\sqrt{m\log p}}\leq -1-\delta\right)=o(1).$$

Combining \eqref{eq:proof1} and \eqref{eq:proof2}, we have
$$\lim_{n\to\infty} \bbP\left(-1\leq\frac{\lambda_{\min}(m)}{2\sqrt{m\log p}}\leq\frac{\lambda_{\max}(m)}{2\sqrt{m\log p}}\leq 1\right)=1,$$
and
$$\lim_{n\to\infty} \bbP\left(1-2\sqrt{\frac{m\log p}{n}}\leq\frac{V_{m,n,p}}{n}\leq\frac{T_{m,n,p}}{n}\leq 1+2\sqrt{\frac{m\log p}{n}}\right)=1,$$
which imply that	
$$\dfrac{T_{m,n,p}}{n}=1+O_p\left(2\sqrt{\frac{m\log p}{n}}\right)$$
and
$$\dfrac{V_{m,n,p}}{n}=1+O_p\left(2\sqrt{\frac{m\log p}{n}}\right).$$\hfill $\square$

\subsection{Proof of Theorem \ref{T2}}

According to Assumption \ref{A2}, we have $\E(x_{11})=0, \E(x_{11}^2)=1$ and $\Var(x_{11}^2)=\eta$. Thus, we have $\E(\xi_t)=0$ and
\begin{align*}
	\E(\xi_t^2)&=\sum_{i=1}^m\E(x_{ti}^2-1)^2u_i^4+4\sum_{i<j}u_i^2u_j^2\\
&=\eta\sum_{i=1}^mu_i^4+4\sum_{i<j}u_i^2u_j^2\\
&=(\eta-2)\sum_{i=1}^mu_i^4+2(\sum_{i=1}^ku_i^2)^2\\
&=(\eta-2)\sum_{i=1}^mu_i^4+2.
\end{align*}

Note that $(\eta-2)\sum_{i=1}^mu_i^4\leq0$ when $0<\eta\leq 2$, and $(\eta-2)\sum_{i=1}^mu_i^4>0$ when $\eta>2$. At the same time, we observe that $\inf_{\|u\|=1}\sum_{i=1}^mu_i^4=\frac{1}{m}$ and $\sup_{\|u\|=1}\sum_{i=1}^mu_i^4=1$.

By Lemma \ref{lemma:chen}, when $x=o(\sqrt{n})$, we have,
\begin{align}\label{eq:proof3}
	\bbP(\| A_S\|\geq x)&\leq \bbP\left(\frac{1}{1-\sqrt{\varepsilon^2(4-\varepsilon^2)}}\sup_{j\leq(1+2/\varepsilon)^m}\vert v_j^\top A_S v_j\vert\geq x\right) \notag \\
	&=\bbP\left(\sup_{j\leq(1+2/\varepsilon)^m} \vert v_j^\top A_S v_j\vert\geq [1-\sqrt{\varepsilon^2(4-\varepsilon^2)}]x\right) \notag \\
    &\leq2(1+2/\varepsilon)^m\sup_{\|v_j\|=1}\bbP\left(v_j^\top A_S v_j\geq [1-\sqrt{\varepsilon^2(4-\varepsilon^2)}]x\right) \notag \\
    &\leq\begin{dcases}
	2\exp\left\{m\log(1+2/\varepsilon)-\frac{[1-\sqrt{\varepsilon^2(4-\varepsilon^2)}]^2x^2}{2[m^{-1}(\eta-2)+2]}\right\}, & 0< \eta\leq2, \\
   2\exp\left\{m\log(1+2/\varepsilon)-\frac{[1-\sqrt{\varepsilon^2(4-\varepsilon^2)}]^2x^2}{2\eta}\right\}, & \eta>2
    \end{dcases}
\end{align}
for sufficiently large $n$.

Considering $0<\eta\leq2$. We first prove
$$\lim_{n\to\infty} \bbP\left(\frac{\lambda_{\max}(m)}{\sqrt{[4(m-1)+2\eta]\log p}}\geq 1+\delta\right)=0.$$
for any $\delta>0$ small enough.

Let $y_n=(1+\delta)\sqrt{[4(m-1)+2\eta]\log p}$. It is easy to see that $y_n=o(\sqrt{n})$. By \eqref{eq:proof3}, we have
\begin{align*}
	&\ \bbP\left(\frac{\lambda_{\max}(m)}{\sqrt{[4(m-1)+2\eta]\log p}}\geq 1+\delta\right)\\
	\leq &\  2\exp\left\{m\log p+m\log(1+2/\varepsilon)-\frac{[1-\sqrt{\varepsilon^2(4-\varepsilon^2)}]^2y_n^2}{2[m^{-1}(\eta-2)+2]}\right\} \\
	=&\ 2\exp\left\{[1-(1+\delta)^2(1-\sqrt{\varepsilon^2(4-\varepsilon^2)})^2]m\log p+m\log(1+2/\varepsilon)\right\}
\end{align*}
for sufficiently large $n$. Similar to the proof of Theorem \ref{T1}, when $\varepsilon^2<2-\sqrt{4-(1-\frac{1}{1+\delta})^2}$, we have $$\bbP\left(\frac{\lambda_{\max}(m)}{\sqrt{[4(m-1)+2\eta]\log p}}\geq 1+\delta\right)=o(1).$$

Next, we also need to show that for any $\delta>0$,
$$\lim_{n\to\infty} \bbP\left(\frac{\lambda_{\min}(m)}{\sqrt{[4(m-1)+2\eta]\log p}}\leq -1-\delta\right)=0.$$

Similarly, we have

\begin{align*}
	&\ \bbP\left(\frac{\lambda_{\min}(m)}{\sqrt{[4(m-1)+2\eta]\log p}}\leq -1-\delta\right) \\
	= &\  \bbP\left(\frac{-\lambda_{\min}(m)}{\sqrt{[4(m-1)+2\eta]\log p}}\geq 1+\delta\right) \\
	\leq &\ 2\exp\left\{m\log p+m\log(1+2/\varepsilon)-\frac{[1-\sqrt{\varepsilon^2(4-\varepsilon^2)}]^2y_n^2}{2[m^{-1}(\eta-2)+2]}\right\} \\
	= &\ 2\exp\left\{[1-(1+\delta)^2(1-\sqrt{\varepsilon^2(4-\varepsilon^2)})^2]m\log p+m\log(1+2/\varepsilon)\right\}
\end{align*}
for sufficiently large $n$. When $\varepsilon^2<2-\sqrt{4-(1-\frac{1}{1+\delta})^2}$, we have $$\bbP\left(\frac{\lambda_{\min}(m)}{\sqrt{[4(m-1)+2\eta]\log p}}\leq -1-\delta\right)=o(1).$$

Hence, we have
$$\lim_{n\to\infty} \bbP\left(-1\leq\frac{\lambda_{\min}(m)}{\sqrt{[4(m-1)+2\eta]\log p}}\leq\frac{\lambda_{\max}(m)}{\sqrt{[4(m-1)+2\eta]\log p}}\leq 1\right)=1,$$
and
$$\lim_{n\to\infty} \bbP\left(1-\sqrt{\frac{[4(m-1)+2\eta]\log p}{n}}\leq\frac{V_{m,n,p}}{n}\leq\frac{T_{m,n,p}}{n}\leq 1+\sqrt{\frac{[4(m-1)+2\eta]\log p}{n}}\right)=1,$$
which imply that	
$$\dfrac{T_{m,n,p}}{n}=1+O_p\left(\sqrt{\frac{[4(m-1)+2\eta]\log p}{n}}\right)$$
and
$$\dfrac{V_{m,n,p}}{n}=1+O_p\left(\sqrt{\frac{[4(m-1)+2\eta]\log p}{n}}\right)$$
when $0<\eta\leq2.$

For $\eta>2$. We first prove
$$\lim_{n\to\infty} \bbP\left(\frac{\lambda_{\max}(m)}{\sqrt{2\eta m\log p}}\geq 1+\delta\right)=0.$$
for any $\delta>0$ small enough.

Let $z_n=(1+\delta)\sqrt{2\eta m\log p}$. It is easy to see that $z_n=o(\sqrt{n})$. By \eqref{eq:proof3}, we have
\begin{align*}
	&\ \bbP\left(\frac{\lambda_{\max}(m)}{\sqrt{2\eta m\log p}}\geq 1+\delta\right)\\
	\leq &\  2\exp\left\{m\log p+m\log(1+2/\varepsilon)-\frac{[1-\sqrt{\varepsilon^2(4-\varepsilon^2)}]^2z_n^2}{2\eta}\right\} \\
	=&\ 2\exp\left\{[1-(1+\delta)^2(1-\sqrt{\varepsilon^2(4-\varepsilon^2)})^2]m\log p+m\log(1+2/\varepsilon)\right\}
\end{align*}
for sufficiently large $n$. When $\varepsilon^2<2-\sqrt{4-(1-\frac{1}{1+\delta})^2}$, we have $$\bbP\left(\frac{\lambda_{\max}(m)}{\sqrt{2\eta m\log p}}\geq 1+\delta\right)=o(1).$$

Finally, we show that for any $\delta>0$,
$$\lim_{n\to\infty} P\left(\frac{\lambda_{\min}(m)}{\sqrt{2\eta m\log p}}\leq -1-\delta\right)=0.$$

By \eqref{eq:proof3}, we have

\begin{align*}
	&\ \bbP\left(\frac{\lambda_{\min}(m)}{\sqrt{2\eta m\log p}}\leq -1-\delta\right) \\
	=&\ \bbP\left(\frac{-\lambda_{\min}(m)}{\sqrt{2\eta m\log p}}\geq 1+\delta\right) \\
	\leq&\ 2\exp\left\{m\log p+m\log(1+2/\varepsilon)-\frac{[1-\sqrt{\varepsilon^2(4-\varepsilon^2)}]^2z_n^2}{2\eta}\right\} \\
	=&\ 2\exp\left\{[1-(1+\delta)^2(1-\sqrt{\varepsilon^2(4-\varepsilon^2)})^2]m\log p+m\log(1+2/\varepsilon)\right\}
\end{align*}
for sufficiently large $n$. When $\varepsilon^2<2-\sqrt{4-(1-\frac{1}{1+\delta})^2}$, we have $$\bbP\left(\frac{\lambda_{\min}(m)}{\sqrt{2\eta m\log p}}\leq -1-\delta\right)=o(1).$$

Hence, we have
$$\lim_{n\to\infty} P\left(-1\leq\frac{\lambda_{\min}(m)}{\sqrt{2\eta m\log p}}\leq\frac{\lambda_{\max}(m)}{\sqrt{2\eta m\log p}}\leq 1\right)=1,$$
and
$$\lim_{n\to\infty} P\left(1-\sqrt{\frac{2\eta m\log p}{n}}\leq\frac{V_{m,n,p}}{n}\leq\frac{T_{m,n,p}}{n}\leq 1+\sqrt{\frac{2\eta m\log p}{n}}\right)=1,$$
which imply that	
$$\dfrac{T_{m,n,p}}{n}=1+O_p\left(\sqrt{\frac{2\eta m\log p}{n}}\right)$$
and
$$\dfrac{V_{m,n,p}}{n}=1+O_p\left(\sqrt{\frac{2\eta m\log p}{n}}\right).$$\hfill $\square$

\subsection{Proof of Theorem \ref{T3}}

For any unit vector $u$, we have
$$u^\top\tilde{W}_Su=\sum_{i=1}^m\tilde{w}_{ii}u_i^2+2\sum_{i<j}\tilde{w}_{ij}u_iu_j\sim N(0,(\eta-2)\sum_{i=1}^mu_i^4+2).$$

Note that $\bbP(N(0,1)\geq x)\leq e^{-x^2/2}$ for all $x\geq 1$. Thus, for $[1-\sqrt{\varepsilon^2(4-\varepsilon^2)}]x\geq 1$,
$$\sup_{\|u\|=1}\bbP\left(u^\top\tilde{W}_Su\geq [1-\sqrt{\varepsilon^2(4-\varepsilon^2)}]x\right)\leq \exp\left\{-\frac{[1-\sqrt{\varepsilon^2(4-\varepsilon^2)}]^2x^2}{2[(\eta-2)\sum_{i=1}^mu_i^4+2]}\right\}.$$

For $0<\eta\leq2$, observe that $(\eta-2)\sum_{i=1}^mu_i^4\leq0$ and $\inf_{\|u\|=1}\sum_{i=1}^mu_i^4=\frac{1}{m}$. Hence,
$$\sup_{\|u\|=1}\bbP\left(u^\top\tilde{W}_Su\geq [1-\sqrt{\varepsilon^2(4-\varepsilon^2)}]x\right)\leq \exp\left\{-\frac{[1-\sqrt{\varepsilon^2(4-\varepsilon^2)}]^2x^2}{2[m^{-1}(\eta-2)+2]}\right\}.$$

For $\eta>2$, observe that $(\eta-2)\sum_{i=1}^mu_i^4>0$ and $\max_{\|u\|=1}\sum_{i=1}^mu_i^4=1$. Hence,
$$\sup_{\|u\|=1}\bbP\left(u^\top\tilde{W}_Su\geq [1-\sqrt{\varepsilon^2(4-\varepsilon^2)}]x\right)\leq \exp\left\{-\frac{[1-\sqrt{\varepsilon^2(4-\varepsilon^2)}]^2x^2}{2\eta}\right\}.$$

Hence, for $\varepsilon\in(0,\sqrt{2-\sqrt{3}})$ and $v_1,\cdots,v_{(1+2/\varepsilon)^m}\in S^{m-1}$, we have
\begin{align*}
\bbP(\|\tilde{W}_S\|\geq x)&\leq P\left(\frac{1}{1-\sqrt{\varepsilon^2(4-\varepsilon^2)}}\sup_{j\leq(1+2/\varepsilon)^m}\vert v_j^\top\tilde{W}_Sv_j\vert\geq x\right)\\
&=\bbP\left(\sup_{j\leq(1+2/\varepsilon)^m}\vert v_j^\top\tilde{W}_Sv_j\vert\geq [1-\sqrt{\varepsilon^2(4-\varepsilon^2)}]x\right)\\
&\leq2(1+2/\varepsilon)^m\sup_{\|v_j\|=1}P\left(v_j^T\tilde{W}_Sv_j\geq [1-\sqrt{\varepsilon^2(4-\varepsilon^2)}]x\right)\\
&\leq2(1+2/\varepsilon)^m\sup_{\|v_j\|=1}\exp\left\{-\frac{[1-\sqrt{\varepsilon^2(4-\varepsilon^2)}]^2x^2}{2[(\eta-2)\sum_{i=1}^mv_{ji}^4+2]}\right\}\\
&=\begin{dcases}
	2(1+2/\varepsilon)^m\exp\left\{-\frac{[1-\sqrt{\varepsilon^2(4-\varepsilon^2)}]^2x^2}{2[m^{-1}(\eta-2)+2]}\right\}, & 0< \eta\leq2, \\
   2(1+2/\varepsilon)^m\exp\left\{-\frac{[1-\sqrt{\varepsilon^2(4-\varepsilon^2)}]^2x^2}{2\eta}\right\}, & \eta>2.
\end{dcases}\\
&=\begin{dcases}
	2\exp\left\{m\log(1+2/\varepsilon)-\frac{[1-\sqrt{\varepsilon^2(4-\varepsilon^2)}]^2x^2}{2[m^{-1}(\eta-2)+2]}\right\}, & 0< \eta\leq2, \\
   2\exp\left\{m\log(1+2/\varepsilon)-\frac{[1-\sqrt{\varepsilon^2(4-\varepsilon^2)}]^2x^2}{2\eta}\right\}, & \eta>2.
\end{dcases}
\end{align*}

The remaining proof is similar to that of Theorem \ref{T2}, and the  details are thus omitted.\hfill $\square$

\backmatter

\bmhead{Supplementary information} The proofs of Corollary \ref{Coro1}, and Corollary \ref{Coro2} are given in the supplementary material.

\bmhead{Acknowledgements} We are very grateful to two anonymous referees, an associate editor, and the editor for their valuable comments that have greatly improved the article.

\section*{Declarations}
\bmhead{Funding} Hu is partially supported by the National Natural Science Foundation of China (nos. 12171187, 11871237).

\bmhead{Conflicts of Interest} The authors have no relevant financial or non-financial interests to disclose.

\bibliography{ref}

\newpage
\titlelabel{S\thetitle.\quad}

\begin{center}
{\Large Supplementary Materials} \\
\vspace{0.5cm}
{\large Jianwei Hu, Seydou Keita, Kang Fu}
\end{center}

In the supplementary materials, we give the proofs of Corollaries 1 and 2. First, we introduce two key quantities. Similar to $\lambda_{\max}(m)$ and $\lambda_{\min}(m)$, we define
\[
\check{\lambda}_{\max}(m)=\max\limits_{S\subset\{1,\cdots,p\},|S|\leq m}\lambda_1(A_S),
\]
and
\[
\check{\lambda}_{\min}(m)=\min\limits_{S\subset\{1,\cdots,p\},|S|\leq m}\lambda_m(A_S),
\]
where $A_S=\dfrac{1}{\sqrt{n}}(X_S^\top X_S-nI)$. It is easy to see that
\[
\check{\lambda}_{\max}(m)=\max_{1\leq k\leq m}\lambda_{\max}(k),\ \text{and}\ \check{\lambda}_{\min}(m)=\min_{1\leq k\leq m}\lambda_{\min}(k).
\]
Further, we also have 
\[
\check{T}_{m,n,p}=n+\sqrt{n}\check{\lambda}_{\max}(m),
\]
and
\[
\check{V}_{m,n,p}=n+\sqrt{n}\check{\lambda}_{\min}(m).
\]

\section{Proof of Corollary 1}

Consider $\check{\lambda}_{\max}(m)$ and $\check{\lambda}_{\min}(m)$. We first prove
$$\lim_{n\to\infty} \bbP\left(\frac{\check{\lambda}_{\max}(m)}{2\sqrt{m\log p}}\geq 1+\delta\right)=0,$$
for any $\delta>0$ small enough.

Let $\check{z}_n=2(1+\delta)\sqrt{m\log p}$, we have $\check{z}_n=o(\sqrt{n})$ and
\begin{align*}
&\ \bbP\left(\frac{\check{\lambda}_{\max}(m)}{2\sqrt{m\log p}}\geq 1+\delta\right)\\
= & \ \bbP\left(\frac{\max\limits_{1\leq k\leq m}\lambda_{\max}(k)}{2\sqrt{m\log p}}\geq 1+\delta\right)\\
\leq &\  \sum_{k=1}^m\bbP\left(\frac{\lambda_{\max}(k)}{2\sqrt{m\log p}}\geq 1+\delta\right)\\
\leq &\  2\sum_{k=1}^m\exp\left\{k\log p+k\log(1+2/\epsilon)-\frac{[1-\sqrt{\epsilon^2(4-\epsilon^2)}]^2\check{z}_n^2}{4}\right\}\\
\leq &\  2\sum_{k=1}^m\exp\left\{[1-(1+\delta)^2(1-\sqrt{\epsilon^2(4-\epsilon^2)})^2]m\log p+m\log(1+2/\epsilon)\right\}\\
\leq &\  2\exp\left\{\log m+[1-(1+\delta)^2(1-\sqrt{\epsilon^2(4-\epsilon^2)})^2]m\log p+m\log(1+2/\epsilon)\right\}
\end{align*}
for sufficiently large $n$.

We also need to show that for any $\delta>0$,
$$\lim_{n\to\infty} \bbP\left(\frac{\check{\lambda}_{\min}(m)}{2\sqrt{m\log p}}\leq -1-\delta\right)=0.$$
We have 
\begin{align*}
&\ \bbP\left(\frac{\check{\lambda}_{\min}(m)}{\sqrt{2m\log p}}\leq -1-\delta\right)\\
= &\  \bbP\left(\frac{-\check{\lambda}_{\min}(m)}{2\sqrt{m\log p}}\geq 1+\delta\right)\\
= &\  \bbP\left(\frac{\max\limits_{1\leq k\leq m}(-\lambda_{\min}(k))}{2\sqrt{m\log p}}\geq 1+\delta\right)\\
\leq &\  \sum_{k=1}^m\bbP\left(\frac{-\lambda_{\min}(k)}{2\sqrt{m\log p}}\geq 1+\delta\right)\\
\leq &\  2\sum_{k=1}^m\exp\left\{k\log p+k\log(1+2/\epsilon)-\frac{[1-\sqrt{\epsilon^2(4-\epsilon^2)}]^2\check{z}_n^2}{4}\right\}\\
\leq &\  2\sum_{k=1}^m\exp\left\{[1-(1+\delta)^2(1-\sqrt{\epsilon^2(4-\epsilon^2)})^2]m\log p+m\log(1+2/\epsilon)\right\}\\
\leq &\  2\exp\left\{\log m+[1-(1+\delta)^2(1-\sqrt{\epsilon^2(4-\epsilon^2)})^2]m\log p+m\log(1+2/\epsilon)\right\}
\end{align*}
for sufficiently large $n$. Notice that it holds that $$1-(1+\delta)^2[1-\sqrt{\epsilon^2(4-\epsilon^2)}]^2<0$$ when $$\epsilon^2<2-\sqrt{4-(1-\frac{1}{1+\delta})^2}.$$ Then, we have $$\bbP\left(\frac{\check{\lambda}_{\max}(m)}{2\sqrt{m\log p}}\geq 1+\delta\right)=o(1)$$ and $$\bbP\left(\frac{\check{\lambda}_{\min}(m)}{2\sqrt{m\log p}}\leq -1-\delta\right)=o(1).$$

Hence, we have
$$\lim_{n\to\infty} \bbP\left(-1\leq\frac{\check{\lambda}_{\min}(m)}{2\sqrt{m\log p}}\leq\frac{\check{\lambda}_{\max}(m)}{2\sqrt{m\log p}}\leq 1\right)=1,$$
and
$$\lim_{n\to\infty} \bbP\left(1-2\sqrt{\frac{m\log p}{n}}\leq\frac{\check{V}_{m,n,p}}{n}\leq\frac{\check{T}_{m,n,p}}{n}\leq 1+2\sqrt{\frac{m\log p}{n}}\right)=1,$$
which imply that
$$\dfrac{\check{T}_{m,n,p}}{n}=1+O_p\left(2\sqrt{\frac{m\log p}{n}}\right)$$
and
$$\dfrac{\check{V}_{m,n,p}}{n}=1+O_p\left(2\sqrt{\frac{m\log p}{n}}\right).$$\hfill $\square$

\section{Proof of Corollary 2}

For $0<\eta\leq 2$, we consider $\check{\lambda}_{\max}(m)$ and $\check{\lambda}_{\min}(m)$. We first prove
$$\lim_{n\to\infty} \bbP\left(\frac{\check{\lambda}_{\max}(m)}{\sqrt{[4(m-1)+2\eta]\log p}}\geq 1+\delta\right)=0,$$
for any $\delta>0$ small enough.

Let $\check{y}_n=(1+\delta)\sqrt{[4(m-1)+2\eta]\log p}$, we have $\check{y}_n=o(\sqrt{n})$ and
\begin{align*}
&\ \bbP\left(\frac{\check{\lambda}_{\max}(m)}{\sqrt{[4(m-1)+2\eta]\log p}}\geq 1+\delta\right)\\
= &\ \bbP\left(\frac{\max\limits_{1\leq k\leq m}\lambda_{\max}(k)}{\sqrt{[4(m-1)+2\eta]\log p}}\geq 1+\delta\right)\\
\leq &\ \sum_{k=1}^m\bbP\left(\frac{\lambda_{\max}(k)}{\sqrt{[4(m-1)+2\eta]\log p}}\geq 1+\delta\right)\\
\leq &\ 2\sum_{k=1}^m\exp\left\{k\log p+k\log(1+2/\epsilon)-\frac{[1-\sqrt{\epsilon^2(4-\epsilon^2)}]^2\check{y}_n^2}{2[m^{-1}(\eta-2)+2]}\right\}\\
\leq &\ 2\sum_{k=1}^m\exp\left\{[1-(1+\delta)^2(1-\sqrt{\epsilon^2(4-\epsilon^2)})^2]m\log p+m\log(1+2/\epsilon)\right\}\\
\leq &\ 2\exp\left\{\log m+[1-(1+\delta)^2(1-\sqrt{\epsilon^2(4-\epsilon^2)})^2]m\log p+m\log(1+2/\epsilon)\right\}
\end{align*}
for sufficiently large $n$.

We also need to show that for any $\delta>0$,
$$\lim_{n\to\infty} \bbP\left(\frac{\check{\lambda}_{\min}(m)}{\sqrt{[4(m-1)+2\eta]\log p}}\leq -1-\delta\right)=0.$$
We have 
\begin{align*}
 &\ \bbP\left(\frac{\check{\lambda}_{\min}(m)}{\sqrt{[4(m-1)+2\eta]\log p}}\leq -1-\delta\right)\\
= &\ \bbP\left(\frac{-\check{\lambda}_{\min}(m)}{\sqrt{[4(m-1)+2\eta]\log p}}\geq 1+\delta\right)\\
= &\ \bbP\left(\frac{\max\limits_{1\leq k\leq m}(-\lambda_{\min}(k))}{\sqrt{[4(m-1)+2\eta]\log p}}\geq 1+\delta\right)\\
\leq &\ \sum_{k=1}^m\bbP\left(\frac{-\lambda_{\min}(k)}{\sqrt{[4(m-1)+2\eta]\log p}}\geq 1+\delta\right)\\
\leq &\ 2\sum_{k=1}^m\exp\left\{k\log p+k\log(1+2/\epsilon)-\frac{[1-\sqrt{\epsilon^2(4-\epsilon^2)}]^2\check{y}_n^2}{2[m^{-1}(\eta-2)+2]}\right\}\\
\leq &\ 2\sum_{k=1}^m\exp\left\{[1-(1+\delta)^2(1-\sqrt{\epsilon^2(4-\epsilon^2)})^2]m\log p+m\log(1+2/\epsilon)\right\}\\
\leq &\ 2\exp\left\{\log m+[1-(1+\delta)^2(1-\sqrt{\epsilon^2(4-\epsilon^2)})^2]m\log p+m\log(1+2/\epsilon)\right\}
\end{align*}
for sufficiently large $n$. Then, when $$\epsilon^2<2-\sqrt{4-(1-\frac{1}{1+\delta})^2},$$ we have $$\bbP\left(\frac{\check{\lambda}_{\max}(m)}{\sqrt{[4(m-1)+2\eta]\log p}}\geq 1+\delta\right)=o(1)$$ and $$\bbP\left(\frac{\check{\lambda}_{\min}(m)}{\sqrt{[4(m-1)+2\eta]\log p}}\leq -1-\delta\right)=o(1).$$

Hence, we have
$$\lim_{n\to\infty} \bbP\left(-1\leq\frac{\check{\lambda}_{\min}(m)}{\sqrt{[4(m-1)+2\eta]\log p}}\leq\frac{\check{\lambda}_{\max}(m)}{\sqrt{[4(m-1)+2\eta]\log p}}\leq 1\right)=1,$$
and
$$\lim_{n\to\infty} \bbP\left(1-\sqrt{\frac{[4(m-1)+2\eta]\log p}{n}}\leq\frac{\check{V}_{m,n,p}}{n}\leq\frac{\check{T}_{m,n,p}}{n}\leq 1+\sqrt{\frac{[4(m-1)+2\eta]\log p}{n}}\right)=1,$$
which imply that	
$$\dfrac{\check{T}_{m,n,p}}{n}=1+O_p\left(\sqrt{\frac{[4(m-1)+2\eta]\log p}{n}}\right)$$
and
$$\dfrac{\check{V}_{m,n,p}}{n}=1+O_p\left(\sqrt{\frac{[4(m-1)+2\eta]\log p}{n}}\right).$$

Next, we consider $\check{\lambda}_{\max}(m)$ and $\check{\lambda}_{\min}(m)$ when $\eta>2$. We first prove
$$\lim_{n\to\infty} P\left(\frac{\check{\lambda}_{\max}(m)}{\sqrt{2\eta m\log p}}\geq 1+\delta\right)=0,$$
for any $\delta>0$ small enough.

Let $\check{z}_n=(1+\delta)\sqrt{2\eta m\log p}$, we have $\check{z}_n=o(\sqrt{n})$ and
\begin{align*}
&\ \bbP\left(\frac{\hat{\lambda}_{\max}(m)}{\sqrt{2\eta m\log p}}\geq 1+\delta\right)\\
= &\ \bbP\left(\frac{\max\limits_{1\leq k\leq m}\lambda_{\max}(k)}{\sqrt{2\eta m\log p}}\geq 1+\delta\right)\\
\leq &\  \sum_{k=1}^m\bbP\left(\frac{\lambda_{\max}(k)}{\sqrt{2\eta m\log p}}\geq 1+\delta\right)\\
\leq &\  2\sum_{k=1}^m\exp\left\{k\log p+k\log(1+2/\epsilon)-\frac{[1-\sqrt{\epsilon^2(4-\epsilon^2)}]^2\check{z}_n^2}{2\eta}\right\}\\
\leq &\  2\sum_{k=1}^m\exp\left\{[1-(1+\delta)^2(1-\sqrt{\epsilon^2(4-\epsilon^2)})^2]m\log p+m\log(1+2/\epsilon)\right\}\\
\leq &\  2\exp\left\{\log m+[1-(1+\delta)^2(1-\sqrt{\epsilon^2(4-\epsilon^2)})^2]m\log p+m\log(1+2/\epsilon)\right\}
\end{align*}
for sufficiently large $n$.

We also need to show that for any $\delta>0$,
$$\lim_{n\to\infty} \bbP\left(\frac{\check{\lambda}_{\min}(m)}{\sqrt{2\eta m\log p}}\leq -1-\delta\right)=0.$$
We have 
\begin{align*}
&\ \bbP\left(\frac{\check{\lambda}_{\min}(m)}{\sqrt{2\eta m\log p}}\leq -1-\delta\right)\\
= &\  \bbP\left(\frac{-\check{\lambda}_{\min}(m)}{\sqrt{2\eta m\log p}}\geq 1+\delta\right)\\
= &\  \bbP\left(\frac{\max\limits_{1\leq k\leq m}(-\lambda_{\min}(k))}{\sqrt{2\eta m\log p}}\geq 1+\delta\right)\\
\leq &\  \sum_{k=1}^m\bbP\left(\frac{-\lambda_{\min}(k)}{\sqrt{2\eta m\log p}}\geq 1+\delta\right)\\
\leq &\  2\sum_{k=1}^m\exp\left\{k\log p+k\log(1+2/\epsilon)-\frac{[1-\sqrt{\epsilon^2(4-\epsilon^2)}]^2\check{z}_n^2}{2\eta}\right\}\\
= &\  2\sum_{k=1}^m\exp\left\{[1-(1+\delta)^2(1-\sqrt{\epsilon^2(4-\epsilon^2)})^2]m\log p+m\log(1+2/\epsilon)\right\}\\
\leq &\  2\exp\left\{\log m+[1-(1+\delta)^2(1-\sqrt{\epsilon^2(4-\epsilon^2)})^2]m\log p+m\log(1+2/\epsilon)\right\}
\end{align*}
for sufficiently large $n$. Then, when $$\epsilon^2<2-\sqrt{4-(1-\frac{1}{1+\delta})^2},$$ we have $$\bbP\left(\frac{\hat{\lambda}_{\max}(m)}{\sqrt{2\eta m\log p}}\geq 1+\delta\right)=o(1)$$ and $$\bbP\left(\frac{\check{\lambda}_{\min}(m)}{\sqrt{2\eta m\log p}}\leq -1-\delta\right)=o(1).$$

Hence, we have
$$\lim_{n\to\infty} \bbP\left(-1\leq\frac{\check{\lambda}_{\min}(m)}{\sqrt{2\eta m\log p}}\leq\frac{\check{\lambda}_{\max}(m)}{\sqrt{2\eta m\log p}}\leq 1\right)=1,$$
and
$$\lim_{n\to\infty} \bbP\left(1-\sqrt{\frac{2\eta m\log p}{n}}\leq\frac{\check{V}_{m,n,p}}{n}\leq\frac{\check{T}_{m,n,p}}{n}\leq 1+\sqrt{\frac{2\eta m\log p}{n}}\right)=1,$$
which imply that	
$$\dfrac{\check{T}_{m,n,p}}{n}=1+O_p\left(\sqrt{\frac{2\eta m\log p}{n}}\right)$$
and
$$\dfrac{\check{V}_{m,n,p}}{n}=1+O_p\left(\sqrt{\frac{2\eta m\log p}{n}}\right).$$\hfill $\square$

\end{document}